\documentclass[10pt]{amsart}

\usepackage{amsmath}

\usepackage{bbm}
\usepackage{enumerate}
\usepackage{paralist}
\usepackage{amssymb}

\usepackage{stmaryrd}

\newtheorem{theorem}{Theorem}

\newtheorem{lemma}[theorem]{Lemma}
\newtheorem{corollary}[theorem]{Corollary}

\theoremstyle{definition}

\newcommand \NN{\mathbbm{N}}

\newcommand \ZZ{\mathbbm{Z}}

\newcommand \Iso{{\rm Iso }}

\newcommand \dist{{\rm dist}}
\newcommand \diam{{\rm diam}}
\newcommand \supp{{\rm supp}}

\newcommand {\gen}[1]{\left\langle #1 \right\rangle}
\newcommand {\clg}[1]{\overline{\left\langle #1 \right\rangle}}

\begin{document}

\title{Abelian pro-countable groups and orbit equivalence relations}
\author{Maciej Malicki}

\address{Institute of Mathematics, Polish Academy of Sciences, Sniadeckich 8, 00-956, Warsaw, Poland, and Lazarski University, Warsaw, Poland}
\email{mamalicki@gmail.com}
 
\begin{abstract}
We study abelian groups that can be defined as Polish pro-countable groups, as non-archimedean groups with an invariant metric or as quasi-countable groups, i.e., closed subdirect products of countable, discrete groups, endowed with the product topology. 

 We show, among other results, that for every non-locally compact, abelian quasi-countable group $G$ there exists a closed $L \leq G$, and a closed, non-locally compact $K \leq G/L$ which is a direct product of discrete, countable groups. As an application we prove that for every abelian Polish group $G$ of the form $H/L$, where $H,L \leq \Iso(X)$ and $X$ is a locally compact separable metric space (e.g., $G$ is abelian, quasi-countable), $G$ is locally compact iff every continuous action of $G$ on a Polish space $Y$ induces an orbit equivalence relation that is reducible to an equivalence relation with countable classes.
\end{abstract}

\maketitle
\section{Introduction}

We study groups that can be characterized in the following three ways. First, they can be defined as Polish, pro-countable groups, i.e., inverse limits of countable families of discrete, countable groups. Also, they can be viewed as non-archimedean groups, i.e., Polish groups with a neighborhood basis at the identity consisting of open subgroups, that admit an invariant metric. Finally, they can be defined as \emph{quasi-countable} groups, i.e., closed, countable subdirect products of discrete, countable groups, endowed with the product topology. In this paper, we will be referring to them as quasi-countable groups because this definition is the most natural, and the most general in case we want to put some additional requirements on their structure.
%
%

As a matter of fact, quasi-countable groups have been studied before (see, e.g., \cite{GaXu}, \cite{HjKe}, \cite{HoMo}.) However, it seems that no one has considered them so far \emph{as} quasi-countable groups, that is, taking advantage of  the fact that they are built out of countable groups in a nice way. Our approach allows for applying tools coming from the theory of discrete groups; for example, we use a classical theorem of Kulikov saying that for every abelian, torsion group $G$ there exists a direct sum of cyclic groups $L \leq G$ such that $K/L$ is divisible.
Most of the results presented in this paper concern non-locally compact, abelian, quasi-countable groups. In particular, we prove a structure theorem to the extent that for every non-locally compact, abelian quasi-countable group $G$ there exists a closed $L \leq G$, and a closed $K \leq G/L$ such that $K$ is a non-locally compact \emph{direct} product of discrete, countable groups. 

 
As an application of our study, we consider connections between the structure of abelian, quasi-countable groups and properties of orbit equivalence relations induced by their continuous actions. Recall that for Polish spaces $X,Y$ and equivalence relations $E,F$ defined on $X$, $Y$, respectively, we say that $E$ is \emph{Borel reducible} to $F$ if there exists a Borel mapping $f: X \rightarrow Y$ such that
\[ xEy \Leftrightarrow f(x)Ff(y) \mbox{ for } x,y \in X \]
(see \cite{Gao} for more information on this notion.) An important class of equivalence relations are orbit equivalence relations, that is, relations of the form 
\[ x E_\alpha y \Leftrightarrow \exists g \in G \, ( \alpha(g,x)=y) \mbox{ for } x,y \in X, \]
where $\alpha$ is a continuous action of $G$ on a Polish space $X$. It turns out that there are deep connections between the structure of $G$ and properties of orbit equivalence relations induced by $G$. For example, S.Solecki \cite{Sol} proved that  a Polish group $G$ is non-compact  if and only if there exists a continuous action $\alpha$ of $G$ such that $E_0$ is reducible to $E_\alpha$ (where $E_0$ is the relation of eventual identity on the space $2^\NN$ of all $0-1$ sequences.) Even more interestingly, A.Kechris  \cite{Ke} proved that if $G$ is locally compact, then every orbit equivalence relation induced by $G$ is reducible to an equivalence relation with countable classes. A natural question arises whether the converse holds. This problem has been remaining 'stubbornly open', as G.Hjorth put it, since the early 1990s.

In this paper we show that the converse does hold for abelian quasi-countable groups. Then, applying certain results of S.Gao, A.Kechris, \\ A.Kwiatkowska, and S.Solecki, as well as the Pontryagin duality, we generalize it to Polish groups of the form $G/L$, where $G,L$ are closed subgroups of $\Iso(X)$ for some locally compact separable metric space $X$.

Let us mention that this partially complements a theorem proved by A.Thompson \cite{Tho}, which says that the converse of Kechris's theorem holds for non-cli groups, that is, Polish groups that do not admit a compatible, left-invariant complete metric. As Polish abelian groups are cli, Thompson's theorem does not apply in the present context.

\section{Notation and basic facts}
A topological space is \emph{Polish} if it is separable and completely metrizable. A topological group is Polish if its group topology is Polish. In the sequel, all discussed groups will be Polish, and all subgroups will be closed (equivalently: Polish.)

It is well known (see \cite[Theorem 2.2.10]{Gao}) that if $G$ is a Polish group, and $H$ is a closed (equivalently: Polish), normal subgroup of $G$, then $G/H$ is a Polish group in the quotient topology. The following fact seems to be standard but we could not find any reference for it.

\begin{lemma}
\label{leComClo}
Let $G$ be a non-locally compact Polish group, and let $H$ be a closed, normal subgroup of $G$. If $H$ is locally compact, then $G/H$ is non-locally compact.
\end{lemma}

\begin{proof}
Suppose that $G/H$ is locally compact. Let $\pi$ be the projection of $G$ onto $G/H$. Fix a compatible left-invariant metric $d$ on $G$, an open neighborhood of the identity $U \subseteq  G/H$ whose closure is compact in $G/H$, and $\epsilon>0$ such that the closure of $B(e,3\epsilon) \cap H$ is compact, where $B(e,3\epsilon)$ is the open ball in $G$ with center at the identity $e$ and radius $3\epsilon$. 

Let $\{x_n \}$ be a sequence in $\pi^{-1}[U] \cap B(e,\epsilon)$. Then $\{ \pi(x_n) \}$ contains a subsequence that converges in $G/H$ to some $gH$ with $d(g,e) \leq \epsilon$. In particular, there exists a subsequence $\{y_n\}$ of $\{x_n\}$ such that $\dist(y_n,gH) \rightarrow 0$. 
Fix $h_n \in H$ such that $d(y_n,gh_n) \rightarrow 0$. Then, for sufficiently large $n$,
\[ d(e,y_n) \leq \epsilon, \, d(y_n,gh_n)=d(g^{-1}y_n,h_n)<\epsilon,\]
so we get
\[ g^{-1}y_n \in B(e,2\epsilon), \, h_n \in B(e,3\epsilon).\]
But the closure of $H \cap B(e,3\epsilon)$ is compact, so there is a convergent subsequence in $\{h_n\}$, and thus a convergent subsequence in $\{y_n\}$.
This implies that $G$ is locally compact, which is a contradiction.


\end{proof}  
 
A group $G$ is called \emph{non-archimedean} if it is Polish, and it has a neighborhood basis at the identity consisting of open subgroups. Equivalently, non-archimedean groups are closed subgroups of the group $S_\infty$ of all permutations of the natural numbers with the pointwise convergence topology.

Let $\mathcal{D}$ be a class of discrete, countable groups. A Polish group $G$ is called \emph{quasi}-$\mathcal{D}$ if it is a subdirect product of groups from $\mathcal{D}$, that is,
\[ G \leq \prod_n G_n, \]
where $G_n \in \mathcal{D}$ for $n \in \NN$, $\prod_n G_n$ is endowed with the product topology, and all the projections of $G$ on $G_n$ are surjective. Thus we can talk about quasi-countable groups, quasi-divisible groups, quasi-dsc groups (where dsc stands for `direct sum of finite cyclic groups'), etc. Similarly, a Polish group $G$ is called pro-$\mathcal{D}$ if it is the inverse limit of an inverse system of groups from $\mathcal{D}$. Clearly, every Polish pro-$\mathcal{D}$ group is quasi-$\mathcal{D}$.
  
Let  $G \leq \prod_n G_n$. For $g \in G$, by $g(n)$ we mean the value of $g$ on its $n$th coordinate, $\pi_n[G]$ denotes the projection of $G$ on $G_n$, $G_{\gen{n}}$ denotes the group
\[ G_{\gen{n}}=\{g \in G: g(k)=e \mbox{ for } k \leq n \},\]
and $G_{\gen{-1}}=G$. Clearly, in the product topology, the family $\{G_{\gen{n}}\}$ forms a neighborhood basis at the identity consisting of normal subgroups of $G$.

If, additionally, $\pi_n[G]=G_n$ for each $n$, we say that the family $\{G_n\}$ is \emph{adequate for} $G$. 

\begin{lemma}
\label{leCharPro}
Let $G$ be a Polish group. The following conditions are equivalent.

\begin{enumerate}
\item $G$ is quasi-countable,
\item $G$ is pro-countable,
\item $G$ has a countable neighborhood basis at the identity consisting of open, normal subgroups,
\item $G$ is non-archimedean, and there exists a compatible two-sided invariant metric on $G$. 
\end{enumerate}
\end{lemma}

\begin{proof}
To show (1)$\Rightarrow$(2), fix a countable family $\{ N_n\}$ of open normal subgroups of $G$ that form a neighborhood basis at the identity, and $N_{n+1} \leq N_n$. Let $f_n$ be the natural projection of $G/N_{n+1}$ onto $G/N_n$. Clearly, $\{G/N_n, f_n)\}$ is as required.

The implication (2)$\Rightarrow$(3) is immediate, and implications (3)$\Rightarrow$(4), (4)$\Rightarrow$(1) follow from the fact that $G$ has a neighborhood basis at the identity consisting of open, normal subgroups iff $G$ is non-archimedean, and admits a compatible invariant metric (see, e.g., \cite[Exercise 2.1.4]{Gao}).
\end{proof}

Let us mention that analogous characterizations can be proved for classes of quasi-countable groups that are sufficiently regular, e.g., quasi-torsion groups or quasi-dsc groups.

An element $g$ of a quasi-countable group $G$ is called pro-$p$, if $\clg{g}$ is a pro-$p$ group. If $G$ is abelian, and $p_0$ is a fixed prime, the $p_0$-\emph{Sylow} subgroup of $G$ is defined as the group of all pro-$p_0$ elements in $G$. This agrees with the standard terminology used in the theory of pro-finite groups (see \cite{RiZa}.)

A cyclic group of order $n$ is denoted by $\mathbbm{Z}(n)$. For a fixed prime $p$, the \emph{Pr\"{ufer} group} $\mathbbm{Z}(p^\infty)$ is the unique $p$-group, in which the number of $p$th roots of every element is exactly $p$. 

For a Polish group $G$, and a continuous action $\alpha$ of $G$ on a Polish space $X$, the symbol $E_\alpha$ denotes the orbit equivalence relation induced on $X$ by $G$ via $\alpha$, that is:
\[ xE_\alpha y \Leftrightarrow \exists g \in G \, (\alpha(g,x)=y) \]
for $x,y \in X$. For two equivalence relations $E$, $F$ defined on Polish spaces $X$, $Y$, respectively, we say that $E$ is \emph{reducible} to $F$, and write $E \leq F$ if there exists a Borel mapping $f:X \rightarrow Y$ such that
\[ xEy \Leftrightarrow f(x)Ff(y).\]
If there exists an injective $f$ as above, we write $E \sqsubseteq F$.

Let $E_0, E_0^\NN$ denote the equivalence relations on $2^\NN$, $(2^\NN)^\NN$, respectively, defined by
\[ x E_0 y \Leftrightarrow \exists m \forall n \geq n \, (x(n)=y(n)), \]
\[ x E_0^\NN y \Leftrightarrow \forall n \, (x(n) E_0 y(n)). \]

It is well known that if $E_0^\NN \leq E$, then $E$ is not reducible to an equivalence relation with countable classes (see, e.g.,  \cite[Theorem 8.5.2]{Gao}, and \cite[Exercise 8.4.3]{Gao}.)

In the sequel we will need two theorems that are important tools in the theory of Borel reducibility.

\begin{theorem}[Mackey, Hjorth, Theorem 3.5.2 in \cite{Gao}]
\label{thMac}
Let $G$ be a Polish group, and let $H$ be a closed subgroup of $G$. Let $\beta$ be a continuous action of $H$ on a Polish space $X$. There exists a Polish space $Y$, and a continuous action $\alpha$ of $G$ on $Y$ such that:
\begin{enumerate}
\item $X$ is a closed subset of $Y$,
\item $\alpha(g, x) = \beta(g, x)$ for all $g \in H$ and $x \in X$,
\item Every $\alpha$-orbit in $Y$ contains exactly one $\beta$-orbit in $X$.
\end{enumerate}

In particular, $E_\beta \sqsubseteq E_\alpha$.
\end{theorem}

\begin{theorem}[Solecki \cite{Sol}]
\label{thSol}
Let $G$ be a non-compact Polish group. There exists a Polish space $X$, and a continuous action $\alpha$ of $G$ on $X$ such that
\[ E_0 \sqsubseteq E_\alpha. \]
\end{theorem}

\section{Main results}

\begin{lemma}
\label{leElProCyc}
Let $G$ be a non-archimedean group, and let $g \in G$. If the group $\clg{g}$ is not discrete, then it is pro-cyclic. In particular, there exist distinct primes $p_k$, $k \in N$, $N \in \NN \cup \{\NN\}$, and pairwise commuting pro-$p_k$ elements $g_k \in \clg{g}$ such that
\[ g=\lim_k g_0 \ldots g_k.\]  

\end{lemma}

\begin{proof}
Put $L=\clg{g}$. Since $L$ is abelian, there exists an invariant metric on $L$. By Lemma \ref{leCharPro}, $L$ is quasi-countable. Without loss of generality we can assume that $L \leq \prod_n L_n$, where each $L_n$ is discrete, and countable. 

Since $\gen{g}$ is not discrete, it is easy to see that $\left\langle g(n)\right\rangle \leq L_n$ must be finite for every $n$. Therefore we can assume that each $L_n$ is finite, that is, $\prod_n L_n$ and $L$ are pro-finite. Since $g$ topologically generates $L$, the group $L$ is pro-cyclic. Now it is well known (see, e.g., \cite[Theorem 2.7.2] {RiZa}) that pro-cyclic groups are direct products of their Sylow subgroups.  
\end{proof}

Let $G_n$, $n \in \NN$, be groups with metrics $d_n$. Define
\[ c_0(G_n,d_n)=\{ g \in \prod_n G_n: d_n(g(n),e) \rightarrow 0 \} \]
with the supremum metric
\[ d_s(g,h)=\sup\{ d_n(g(n),h(n)): n \in \NN \}. \]

\begin{lemma}
Let $G_n$, $n \in \NN$, be Polish groups with invariant, compatible metrics $d_n$. The group $c_0(G_n,d_n)$ with metric $d_s$ is Polish.
\end{lemma}

\begin{proof}
Fix countable, dense sets $D_n \subseteq G_n$ for each $n \in \NN$. It is easy to see that the set
\[ D=\{ g \in c_0(G_n,d_n): \ \forall n \, g(n) \in D_n \mbox{ and } \forall^\infty n \, g(n)=e \}. \]
is countable and dense in $c_0(G_n,d_n)$. 

Fix now a Cauchy sequence $\{g_m \}$ in $c_0(G_n,d_n)$. Clearly, $\{g_m(n)\}$ is Cauchy in $d_n$ for every fixed $n$. Because the metrics $d_n$ are complete (see \cite[Exercise 2.2.4]{Gao}),  for every fixed $n$ there exists $g(n) \in G_n$ such that $g_m(n)$ converge to $g(n)$. Because the sequence $\{g_m\}$ is Cauchy in $d_s$,  the element $g=(g(0),g(1), \ldots )$ is in $c_0(G_n,d_n)$, and $\{g_m\}$ converges to $g$ in $d_s$. Therefore $d_s$ induces a Polish topology on $c_0(G_n,d_n)$.

In order to see that group operations are continuous in $d_s$, note that $d_s$ is also invariant. Therefore (see, e.g., \cite[Exercise 2.1.6]{Gao})
\[ d_s(gh,g'h') \leq d_s(g,h)+d(g',h') \]
for every $g,g',h,h' \in c_0(G_n,d_n)$. This easily implies that multiplication is continuous with respect to $d_s$. Continuity of inversion directly follows from the invariance of $d_s$.
\end{proof}

\begin{lemma}
\label{le7}
Let $G$ be an abelian quasi-torsion group with an invariant metric $d$. Let $\{H_m\}$ be the family of all Sylow subgroups of $G$. Then $G=c_0(H_m,d_m)$, where $d_m$ is the restriction of $d$ to $H_m$.
\end{lemma}

\begin{proof}
%
%

We show that the mapping $h \mapsto \sum_m h_m$ is a well defined isomorphism $\phi: c_0(H_m,d_m) \rightarrow G$. 
Fix $h_m \in H_m$ such that $h_m \rightarrow 0$. Fix $\epsilon>0$, and an open subgroup $V \leq G$ such that $\dist(V,0)< \epsilon$. Let $m_0$ be such that $h_m \in V$ for $m \geq m_0$. Then 
\[ \sum_{m=m_0}^{m_1} h_m \in V \]
for every $m_1>m_0$, which implies
\[ d(h_0+ \ldots + h_{m_1}, h_0+ \ldots + h_{m_0})=d(h_{m_0+1}+ \ldots + h_{m_1},0)< \epsilon. \]
Since $\epsilon$ was arbitrary, $\sum_m h_m$ is convergent, and $\phi$ is well defined. An analogous argument shows that $\phi$ is a continuous homomorphism.

Let $\{G_n\}$ be an adequate family for $G$ consisting of torsion groups. Observe that if $h_m \in H_m$,  $m=0, \ldots, l$, then
\[ h_0(n)+ \ldots+h_l(n) = 0 \ \Leftrightarrow h_0(n)= \ldots =h_l(n) = 0 \]
for every $n$ (here, $h_m(n)$ is the value of the $n$th coordinate of $h_m$ in $\prod_n G_n$.) This easily implies that for $h_m \in H_m$, $m \in \NN$, such that $ h_m \rightarrow 0$ we have that
\[ \sum_m h_m=0 \ \Leftrightarrow h_m=0 \mbox{ for every } m, \] 
that is, $\phi$ is injective. By Lemma \ref{leElProCyc}, for every $g \in G$ there exist $h_m \in H_m$, $m \in \NN$, such that $g=\sum_m h_m$, that is, $\phi$ is also surjective. But this implies that $\phi$ is open (see \cite[Theorem 2.3.3]{Gao}), so $\phi$ is an isomoprhism.
\end{proof}

The same argument can be used to prove the following.

\begin{lemma}
\label{le:c0}
Let $G$ be an abelian quasi-torsion group with an invariant metric $d$, and let $H_0, H_1, \ldots$ be subgroups of $G$ such that
\begin{enumerate}
\item $\diam(H_n) \rightarrow 0$,
\item each $H_n$ is a pro-$\pi_n$-group for some set of primes $\pi_n$, 
\item $\pi_n \cap \pi_{n'}= \emptyset$, if $n \neq n'$.
\end{enumerate}
Then $\overline{\bigoplus_n H_n}=\prod_n H_n \leq G$.
\end{lemma}

\begin{proof}
Observe that $c_0(H_n,d_n)=\prod_n H_n$, where $d_n$ is the restriction of $d$ to $H_n$, and prove that $\phi:\prod_n H_n \rightarrow \overline{\bigoplus_n H_n}$ defined by $h \mapsto \sum_n h(n)$ is an isomorphism.
\end{proof}

Now we prove two special cases of the main structure theorem which are of some independent interest.

\begin{lemma}
\label{leLocCom}
Let $G$ be an abelian quasi-torsion group. Suppose that $G$ is non-locally compact, and all Sylow subgroups of $G$ are locally compact. Then there exist non-compact groups $K_n$, $n \in \NN$, such that 
%
%
%
\[ \prod_n K_n \leq G.\]
\end{lemma}

\begin{proof}
Let $\{G_n\}$ be an adequate family for $G$ consisting of torsion groups. Because $G$ is non-locally compact, we can assume that $G_n$ are such that $\pi_n[G_{\gen{m}}]$ is infinite for every $m$ and $n>m$. Let $\{H_m\}$ be the family of all Sylow subgroups of $G$. 
We consider the following two cases.

\emph{Case 1.} For every $n$ there exists  $n' \geq n$ and $m$ such that $H_m \cap G_{\gen{n'}}$ is non-compact. Observe that since each $H_m$ is locally compact, for every $m$ there exist only finitely many $n$ such that $H_m \cap G_{\gen{n}}$ is non-compact. Therefore we can find an infinite sequence $(k_n,l_n)$, $n \in \NN$, such that $k_n \neq k_{n'}$ if $n \neq n'$, $( l_n)$ is strictly increasing, and
\[ K_n=H_{k_n} \cap G_{\gen{l_n}}\]
is non-compact. Lemma \ref{le:c0} implies that $\prod_n K_n \leq G$. 
%
%

\emph{Case 2.} There exists $n$ such that $G_{\gen{n}} \cap H_m$ is compact for every $m$. After replacing $G$ with $G_{\gen{n}}$, which is also non-locally compact, we can assume that every $H_m$ is compact.

Suppose that there exists an infinite set $A \subseteq \NN$ such that for every infinite $B \subseteq A$ there exist infinitely many $n$ such that 
\[ \overline{\bigoplus_{m \in B} H_m} \cap G_{\gen{n}}\]
 is non-compact. Then we can partition $A$ into infinite sets $B_n$, and find a strictly increasing sequence $(l_n)$ such that for
\[ K_n= \overline{\bigoplus_{m \in B_n} H_m} \cap G_{\gen{l_n}},\]
each $K_n$ is non-compact, and, as in Case 1, $\prod_n K_n \leq G$.

Otherwise, we put 
\[ A_0=\{ m \in \NN : \pi_0[H_m] \mbox{ is not trivial} \}. \]

By our assumption on the form of the adequate family $\{ G_n \}$, and by compactness of all $H_m$, the set $A_0$ is infinite. Obviously for every infinite $B \subseteq A_0$ we have then that 
\[ K_0=\overline{\bigoplus_{m \in B} H_m} \]
is non-compact. On the other hand, since there is no set $A \subseteq \NN$ as above, there must exist an infinite $B_0 \subseteq A_0$ and $l_0$ such that 
\[ \overline{\bigoplus_{m \in B_0} H_m} \cap G_{\gen{l_0}}\]
is compact. Then, again, the set $A_1$ defined by 
\[ A_1=\{ m \in \NN \setminus B_0 : \pi_{l_0+1}[H_m \cap G_{\gen{l_0}}] \mbox{ is not trivial} \}, \]
is infinite, and the group
\[ \overline{\bigoplus_{m \in B} H_m} \cap G_{\gen{l_0}}\]
is non-compact for every infinite $B \subseteq A_1$. We can find $B_1 \subseteq A_1$ and $l_1 \in \NN$ in the same way we have found $B_0$ and $l_0$.

In this manner, we construct pairwise disjoint infinite sets $B_n \subseteq \NN$, $n \in \NN$, and a strictly increasing sequence $(l_n)$ so that 
\[ K_n=\overline{\bigoplus_{m \in B_n} H_m} \cap G_{\gen{l_n}}\]
is non-compact for every $n$. An application of Lemma \ref{le:c0} completes the proof.
\end{proof}

Since each $K_n$ in Lemma \ref{leLocCom} is a non-compact, quasi-countable group, we can find closed $L_n \leq K_n$ such that $K_n /L_n$ is infinite and discrete. Therefore we get the following corollary.

\begin{corollary}
Let $G$ be an abelian quasi-torsion group. Suppose that $G$ is non-locally compact, and all Sylow subgroups of $G$ are locally compact. Then there exists a closed $L \leq G$, and infinite, discrete groups $K_n$, $n \in \NN$, such that 
%
%
%
\[ \prod_n K_n \leq G/L.\]
\end{corollary}

\begin{lemma}
\label{leTor}
Let $G$ be an abelian and torsion quasi-countable group. If $G$ is non-locally compact, then there exists an infinite, discrete subgroup $L \leq G$.
\end{lemma}

\begin{proof}
%
%
Let $\{G_n\}$ be an adequate family for $G$, and for $n,m \in \NN \setminus \{0\}$, let
\[ A_{n,m} =\{ g \in G: \, o(g) \leq m \mbox{ and }  \forall k \leq m (kg \in G_{\gen{n}} \rightarrow kg=0) \}. \]
 
Observe that each $A_{n,m}$ is closed. Since $G$ is torsion, $\bigcup_{n,m} A_{n,m}=G$, and there exist $n_0,m_0$ such that $A_{n_0,m_0}$ is not meager.  Let $H \leq G$ be the subgroup generated by all elements of order $\leq m_0$ in $G$.

Note that non-local compactness of $G$ implies that if  $A_{n_0,m_0} \subseteq H$ was compact, then it would be meager. Therefore there must exist $n_1 \geq n_0$ such that $A_{n_0,m_0} / G_{\gen{n_1}}$ is an infinite set.

We construct $L$ now. Put $L_0=\gen{h}$ for some fixed element $h \in A_{n_0,m_0}$, and suppose that a finite group $L_k$ has been already constructed so that 
\[ L_k \cap G_{\gen{n_1}} = \{0\}.\]

The group $L_k/G_{\gen{n_1}}$ is finite, while $A_{n_0,m_0}/G_{\gen{n_1}}$ is infinite. Moreover, $H/G_{\gen{n_1}}$ has bounded exponent, so it is a direct sum of finite cyclic groups (see, e.g., \cite[Theorem 17.2]{Fu}.) Therefore there is $h \in A_{n_0,m_0} \setminus L_k$ such that
$\langle L_k,h \rangle$ is isomorphic to $L \oplus \langle h \rangle$, and, by the definition of the sets $A_{n,m}$,
\[\left\langle h \right\rangle \cap (L_k+G_{\gen{n_1}}) = \{0 \}.\]

But then
\[ \left\langle L_k,h \right\rangle \cap G_{\gen{n_1}}=\{0\}, \]
so we can put $L_{k+1}=\gen{L_k,h}$, $L=\bigcup_k L_k$. Clearly, $L$ is infinite and discrete.
\end{proof}

%


\begin{corollary}
\label{leTorNonProTor}
Let $G$ be an abelian quasi-countable group which is torsion or such that in every neighborhood of $0$ there exists an element generating an infinite, discrete group. If $G$ is non-locally compact, there exist infinite, discrete $L_n \leq G$, $n \in \NN$, such that
\[ \prod_n L_n \leq G.\]

\end{corollary}

\begin{proof}
Let $\{G_n \}$ be an adequate family for $G$. By Lemma \ref{leTor}, there exists an infinite, discrete $K_n \leq G_{\gen{n}}$ for every $n \in \NN$. For every $n$ fix $f_0(n) \in \NN$ such that
\[ K_n \cap G_{\gen{f_0(n)}} =\{0\},\]
and put $f(n)=(f_0)^n(0)$. Clearly,
\[ K_{f(n)} \cap K_{f(n')} = \{0\}\]
if $n \neq n'$, and the sum $\sum_n h(n)$ is convergent for every $h \in \prod_n K_{f(n)}$. A proof similar to that of Lemma \ref{le7} shows that every convergent sequence in $\overline{\bigoplus_n K_{f(n)}}$ is of the form $\sum_n h(n)$ for a unique $h \in \prod_n K_{f(n)}$. Thus, we can put $L_n=K_{f(n)}$.
\end{proof}

Let us state Kulikov's theorem.

\begin{theorem}[Kulikov]
Let $G$ be an abelian, torsion group. There exists a direct sum of cyclic groups $L \leq G$ such that $G / L$ is divisible.
\end{theorem}

\begin{lemma}
\label{leKul}
Suppose that $G$ is an abelian, quasi-torsion group. If $G$ is non-locally compact, then either there exists $n$ such that $G_{\gen{n}}$ is quasi-dsc or there exists a closed subgroup $L \leq G$ such that $G/L$ is quasi-divisible and non-locally compact.
\end{lemma}

\begin{proof}
Let $\{G_n \}$ be an adequate family for $G$ consisting of torsion groups. 
Applying Kulikov's theorem, for each $n$ fix a direct sum of cyclic groups $L_n$ such that $G_n/L_n$ is divisible. By \cite[\S 18, Theorem 1]{Fu}, if $A$ is a torsion, abelian group, $B \leq A$ is a dsc group, and $A/B$ is finite, then $A$ is a dsc group. Therefore there are two cases to consider:

\emph{Case 1.} There exists $m$ such that $\pi_n[G_{\gen{m}}]/L_n$ is finite for every $n \geq m$. By the above observation, $\pi_n[G_{\gen{m}}]$ is a direct sum of cyclic groups for every $n \geq m$, that is, $G_{\gen{m}}$ is quasi-dsc.

\emph{Case 2.} For every $m$ there exists $n >m$ such that $\pi_n[G_{\gen{m}}] / L_n$ is infinite. Then $G / \prod_n L_n$ is non-locally compact. Put $L=\prod_n L_n$.
\end{proof}

In the sequel we assume that every direct sum of finite cyclic groups $G$ comes equipped with a fixed basis $\{x_m\}$, that is, a family of elements $x_m \in G$ such that
\[ G= \bigoplus_m \gen{x_m}. \]

%
%
For $x \in G$, $A \subseteq G$, we define 
\[ \supp_G(x)=\{m \in \NN: \mbox{ the projection of } x \mbox{ onto } \gen{x_m} \mbox{ is non-trivial} \},\]
\[ \supp_G(A)=\bigcup\{ \supp(x): x \in A \}, \]
and, for $E \subseteq \NN$,
\[ x \upharpoonright E =\pi(x),   \,  A \upharpoonright E=\pi[A]\]
where $\pi$ is the projection of $G$ onto $\bigoplus_{m \in E} \gen{x_m}$.
%


%
%
%
%

\begin{lemma}
\label{thDsc}
Let $G$ be a non-locally compact, abelian, quasi-dsc group. There exist an adequate family $\{G_n\}$ for $G$,  $L_n \leq G_n$, and infinite $K_n \leq G_n$, $n \in \NN$, such that 

\begin{enumerate}
\item each $L_n$ is a direct summand in $G_n$,
\item $\prod_n K_n \leq G / \prod_n L_n$.
\end{enumerate}
\end{lemma}

\begin{proof}
Because $G$ is non-locally compact, by combining factors we can fix an adequate family $\{ G_n \}$ such that $\pi_{n+1}[G_{\gen{n}}]$ is infinite for every $n$. 

Fix a sequence $\{n_k\}$ of natural numbers such that each natural number appears in it infinitely many times.
We will find a sequence $g_k \in G$, $k \in \NN$, such that for each $k$:

\begin{enumerate}[(a)]
\item $g_k \in G_{\gen{n_k-1}}$,
\item $g_k(n_k) \neq 0$,
\item $\supp_{G_n}(g_{k}(n)) \cap \supp_{G_n}(g_{k'}(n))= \emptyset$ for any $n \leq k$ and $k' < k$.
\end{enumerate}

For $k=0$ fix any $g_0 \in G_{\gen{n_0-1}}$ with $g_0(n_0) \neq 0$, and suppose that $g_0, \ldots, g_{l-1} \in G$ satisfying Conditions (a)-(c) for $k<l$ have been already constructed.

For every $n<l$ define

\[ E_n=\bigcup_{k<l} \supp_{G_n}(g_k(n)).\]
%
For every $m \in \NN$ fix $h_m \in G_{\gen{n_l-1}}$ such that $h_m(n_l) \neq h_{m'}(n_l)$ if $m \neq m'$.
%
%
Since each $E_n$ is finite, by the pigeon hole principle there exist distinct $m,m' \in \NN$ such that
%

\[ h_{m}(n) \upharpoonright E_n = h_{m'}(n) \upharpoonright E_n \]
%
for $n<l$. 
%
%
%
Let $g_l=h_{m}-h_{m'}$. It is easy to see that $g_l$ satisfies Conditions (a)-(c).
 
%



For $n \in \NN$ define
\[ K_n=\gen{\{g_{k}(n): k \in \NN \mbox{ such that } n_k=n\}}, \]
\[ F_n= \supp_{G_n}(K_n), \]  
\[ L_n=G_n \upharpoonright (\NN \setminus F_n). \]

By Conditions (b) and (c), the group $K_n$ is an infinite subgroup of $G_n$. Moreover, obviously,
\[ \supp_{G_n}(K_n) \cap \supp_{G_n}(L_n)=\emptyset,\]
and $L_n$ is a direct summand in $G_n$.

We show that for every $h \in \prod_n K_n$ we can find $g \in G$ such that for every $n$
\[ g(n) \upharpoonright F_n=h(n).\]

Fix $h \in \prod_n K_n$.  For each $n$, find $k_0, \ldots, k_m \in \NN$, $l_0, \ldots l_m \in \NN$ such that
\[ n_{k_i}=n, \]
\[ h(n)=l_0 g_{k_0}(n)+ \ldots+ l_m g_{k_m}(n),\]
and put
\[ 
\gamma_n=l_0 g_{k_0}+ \ldots +l_m g_{k_m}.\]
%
%
%

Note that Condition (c) implies that
\[ \supp(g_k(n)) \cap F_n=\emptyset \]
for any $k$ with $n_k \neq n$. 
Also, Condition (a) gives that $\gamma_n \in G_{\gen{n-1}}$, so the limit $\gamma=\sum_n \gamma_n$ exists. Therefore
\[ \gamma(n) \upharpoonright F_n=h(n) \]
for every $n$. Since $h \in \prod_n K_n$ was arbitrary, we get
\[ \prod_n K_n \leq G/\prod_n L_n. \] 

\end{proof}

\begin{theorem}
\label{thProP}
Suppose that $G$ is an abelian, quasi-torsion group. If $G$ is non-locally compact, then there exists a closed $L \leq G$, and infinite, discrete groups $K_n$, $n \in \NN$, such that 
\[ \prod_n K_n \leq G / L. \]
\end{theorem}

\begin{proof}
If there exists $n$ such that $G_{\gen{n}}$ is quasi-dsc, we can use Lemma \ref{thDsc}. Otherwise, we apply Lemma \ref{leKul} to find $L \leq G$ such that $H=G/L$ is quasi-torsion, quasi-divisible and non-locally compact. Let $\{H_n\}$ be an adequate family for $H$ consisting of torsion, divisible groups. Because $H$ is non-locally compact, we can assume that $\pi_m[H_{\gen{n}}]$ is infinite for every $n$ and $m > n$.

It is well known that every torsion, divisible group is isomorphic to a direct sum of  Pr\"{u}fer groups. Suppose that for every $n$ there exists $f(n)>n$, and a Pr\"{u}fer group $K'_{f(n)}$ which is a direct summand in $H_{f(n)}$, and is such that the projection of $\pi_{f(n)}[H_{\gen{n}}]$ on $K'_{f(n)}$ is surjective. Let $L'_{f(n)} \leq H_{f(n)}$ be such that
\[ H_{f(n)}=K'_{f(n)} \oplus L'_{f(n)}.\]

Put $K_{f^m(0)}=K'_{f^m(0)}$ for $m \in \NN$, and put $L_{f^m(0)}=L'_{f^m(0)}$ for $m \in \NN$, $L_m=H_m$ if $m$ is not of the form $f^n(0)$ for some $n$.
%
%
It is easy to see that 
\[ H/\prod_n L_n = \prod_n K_n, \]
and $\prod_m K_m$ is non-locally compact.

Otherwise, there exists $m$ such that for every $n>m$, the projection of $\pi_n[H_{\gen{m}}]$ on every Pr\"{u}fer group which is a direct summand in $H_n$, is a finite, cyclic group. Then $\pi_n[H_{\gen{m}}]$ is a subgroup of a direct sum of cyclic groups. But it is known (see, e.g., \cite[Theorem 18.1]{Fu}) that a subgroup of a direct sum of cyclic groups is also a direct sum of cyclic groups, so $H_{\gen{m}}$ is quasi-dsc. Again, we can apply Lemma \ref{thDsc}.
\end{proof}

\begin{theorem}
\label{thProC}
Let $G$ be an abelian quasi-countable group. If $G$ is non-locally compact, there exists a closed $L \leq G$, and infinite, discrete  groups $K_n$, $n \in \NN$, such that 
\[ \prod_n K_n \leq G / L. \]

Moreover, if $G$ is torsion or in every neighborhood of $0$ there exists an element generating an infinite, discrete group, we can put $L=\{0\}$.
\end{theorem}

\begin{proof}
If $G$ is torsion or  in every neighborhood of $0$ there exists an element generating an infinite, discrete group, we can apply Corollary \ref{leTorNonProTor}. Otherwise, by Lemma \ref{leElProCyc}, there exists an open subgroup $H$ in $G$ all of whose elements are pro-cyclic, which means that $H$ is quasi-torsion. Since $H$  is also non-locally compact, we can assume that $G$ itself is quasi-torsion, and apply Theorem \ref{thProP}. 
\end{proof}

Next, we would like to analyse what happens, if we require that $L$ in Theorem \ref{thProC} is supposed to be a pro-cyclic group. It turns out that this gives some more insight in the structure of abelian quasi-$p$ groups. Let us start with some auxiliary results.
 
\begin{lemma}
Let $G$ be a pro-$p$, pro-cyclic group, let $g \in G$ be a topological generator of $G$, and let $H \leq G$ be closed and non-trivial. Then $ng \in H$ for some $n>0$.
\end{lemma}

\begin{proof}
By our assumption,
\[ G=\varprojlim \ZZ(p^n), \]
where $\ZZ(p^n)$ is the cyclic group of order $p^n$. For $n \in \NN$, let $\pi_n$ be the projection of $G$ onto its $n$th coordinate.
Observe that for any closed $H_1, H_2 \leq G$ either $H_1 \leq H_2$ or $H_2 \leq H_1$. Indeed,
\[ \pi_n(H_1) \leq \pi_n(H_2) \mbox{ or } \pi_n(H_2) \leq \pi_n(H_1) \]
for every $n \in \NN$, and 
\[ \pi_{n+1}(H_1) \leq \pi_{n+1}(H_2) \rightarrow \pi_n(H_1) \leq \pi_n(H_2). \]
By closedeness of $H_1$ and $H_2$, the claim follows.

Note also that
\[ \bigcap_n \overline{\left\langle ng\right\rangle}=\{0\}, \]
which implies that if $H \neq \{0\}$, then there exists $n>0$ such that $\overline{\left\langle ng\right\rangle} \leq H$. Since $H$ is closed, we get that $ng \in H$.

\end{proof}

\begin{corollary}
\label{coPom1}
Suppose that $G$ is an abelian quasi-$p$ group, and $g,g' \in G$. Then
\[ \clg{g} \cap \gen{g'} = \{0\} \]
implies that
\[ \clg{g} \cap \clg{g'}=\{0\}. \]
\end{corollary}
 

%
%
%
%
%
%
%


\begin{theorem}
Let $G$ be an abelian quasi-$p$ group. Then one of the following holds:

\begin{enumerate}
\item There exists a pro-cyclic $H \leq G$ such that $G/H$ contains a clopen subgroup with bounded exponent.
\item The set $D \subseteq G^2$ defined by 
\[  (g,g') \in D \leftrightarrow \overline{\left\langle g \right\rangle} \cap \overline{\left\langle g' \right\rangle}= \{0\}\]
is comeager in $G^2$.
\end{enumerate}
\end{theorem}

\begin{proof}

If the group of torsion elements is non-meager in $G$, there exists $n$ such that the group generated by elements of order $\leq n$ is non-meager, and hence open in $G$. In this case, Point (1) holds for $H=\{0\}$.

%

If there exists $g \in G$ such that $H=\clg{g}$ is compact, and for non-meager many $g' \in G$ there is $n>0$ such that $ng' \in H$, then $G/H$ contains a clopen subgroup with bounded exponent. 

Otherwise, we have
\[ \forall g\in G \, \forall^* g' \in G  \, (\clg{g} \cap \gen{g'} = \{0 \}), \]
so, by Corollary \ref{coPom1},
\[ \forall g\in G \, \forall^* g' \in G  \, (\clg{g} \cap \clg{g'}= \{0 \}). \]

Then, by the Kuratowski-Ulam theorem, the set $D$, defined as in the statement of the lemma, is comeager.

%
%
%

%

%
%
%
%
\end{proof}

\section{Applications}

\begin{theorem}
\label{thAcPro}

Suppose that $G$ is an abelian quasi-countable group. Then $G$ is locally compact if and only if every continuous action of $G$ on a Polish space $X$ is reducible to an equivalence relation with countable classes.
\end{theorem}

\begin{proof}
The implication from left to right follows from \cite{Ke}. To show the other implication, fix $K_n$, $n \in \NN$, and $L$ as in Theorem \ref{thProC}.
By Theorem \ref{thSol}, there exist continuous actions $\alpha_n$ of $K_n$ on some Polish spaces $Y_n$ such that $E_0 \sqsubseteq E_{\alpha_n}$. Then
\[ E_0^\NN \sqsubseteq E_{\alpha''}, \]
where $\alpha''$ is the action of $\prod_n K_n$ on $\prod_n Y_n$ defined by
\[ \alpha''(g,y)=(\alpha_n(g(n),y(n))) \]
for $g \in \prod_n K_n$, $y \in \prod_n Y_n$. By Theorem \ref{thMac}, the action $\alpha''$ can be extended to a continuous action $\alpha'$ of $G/ L$ on some Polish space $X$ so that
\[ E_{\alpha''} \sqsubseteq E_{\alpha'}. \]
Define $\alpha$ as
\[ \alpha(g,x)=\alpha'(g/L,x), \]
Then 
\[ E_0^\NN \sqsubseteq E_\alpha,\]
so $E_\alpha$ is not reducible to an equivalence relation with countable classes.

%
%
%
%


%

%
\end{proof}

It is well known that every Polish group is isomorphic to some $G \leq \Iso(X)$, where $X$ is a Polish metric space, and $\Iso(X)$ is the group of all isometries of $X$ with the topology of pointwise convergence. Using the Pontryagin duality, a characterization of isometry groups of locally compact spaces due to S.Gao and A.Kechris, and a result of A. Kwiatkowska and S.Solecki, we are able to show the following.

\begin{theorem}
Let $X$ be a locally compact separable metric space, and let $G,L \leq \Iso(X)$ be Polish groups such that $G/L$ is abelian. Then $G/L$ is locally compact if and only if every continuous action of $G/L$ on a Polish space $X$ is reducible to an equivalence relation with countable classes.
\end{theorem}

\begin{proof}
By \cite[Corollary 1.3]{SolKw}, there exists a locally compact separable metric space $X$ such that $G/L \leq \Iso(X)$, so we can assume that $L$ is trivial and $G$ is abelian. 
As before, we will prove the implication from right to left by showing that if $G$ is non-locally compact, then there exists a continuous action $\alpha$ of $G$ such that
\[ E_0^\NN \sqsubseteq E_\alpha.\]

By \cite[Theorem 6.3]{KeGao}, there exist closed, subgroups $K_n \leq S_\infty$, and locally compact groups $L_n$,  $n \in \NN$, such that
\[ G \leq \prod_n (K_n \ltimes L_n^\NN), \]
where each $K_n$ acts on $L_n^\NN$ by permuting coordinates. Since $G$ is abelian, the definition of the semidirect product gives that we can assume without loss of generality that all the groups $K_n$ are abelian.

Fix a neighborhood basis $\{U_n\}$ at the identity in $G$, and let $H_n$, $n \in \NN$, be the clopen groups generated by $U_n$. If 
\[ H=\bigcap_n H_n, \]
is locally compact, then, by Lemma \ref{leComClo}, $G/H$ is non-locally compact, and the family $\{(G/H)/H_n)\}$ witnesses that $G/H$ is an abelian quasi-countable group. Thus, we can apply Theorem \ref{thAcPro} to $G/H$. As every action of $G/H$ gives rise to an action of $G$, this finishes the proof of the case that $H$ is locally compact.

Otherwise, $H \leq G$ is non-locally compact, and, clearly, the actions of $K_n$ on $L_n^\NN$ are trivial when restricted to elements of $H$. Therefore
\[ H \leq (\prod_n K_n) \oplus (\prod_n L_n), \]
and, since $G$ is abelian, we can assume that $L_n$ are abelian as well.

If the projection of $H$ on $\prod_n K_n$ is non-locally compact, we apply Theorem \ref{thAcPro}. Otherwise, the projection of $H$ on $\prod_n L_n$ must be non-locally compact. It is a well known result in the theory of locally compact abelian groups, following from the Pontryagin duality, that every such group has an open subgroup topologically isomorphic to $\mathbbm{R}^k \oplus C$ for some $k\geq 0$ and compact group $C$ (see, e.g., \cite[Theorem 25]{Mor}). For every $L_n$, fix such $R_n=\mathbbm{R}^k$, and $C_n=C$.

Suppose that $H \cap \prod_n R_n$ is locally compact. Then $H/ \prod_n R_n$ is non-locally compact, and, since all $C_n$ are compact, $H/ \prod_n(R_n \oplus C_n)$ is non-locally compact as well. But each $R_n \oplus C_n$ is open in $L_n$, so $L_n / (R_n \oplus C_n)$ is countable. In other words, $H/ \prod_n(R_n \oplus C_n)$ is a quasi-countable group, and we can apply Theorem \ref{thAcPro} once more.

The last possibility to consider is that $H_0=H \cap \prod_n R_n$ is non-locally compact. By \cite[Theorem 2]{BrHiMo},
\[ H_0 \cong \mathbbm{Z}^k \oplus \mathbb{R}^l, \]
where, $k,n \in \NN \cup \{\NN\}$, and, by non-local compactness of $H_0$, at least of of them is $\NN$. Applying Theorem \ref{thSol} to each element of the infinite product, we can find an action $\beta$ of $H_0$ such that
\[ E_0^\NN \sqsubseteq E_\beta. \]

Using Theorem \ref{thMac}, we extend $\beta$ to an action $\alpha$ of $G$ so that
\[ E_0^\NN \sqsubseteq E_\alpha. \]
\end{proof}



\end{document}